\RequirePackage{fix-cm}
\documentclass[11pt, a4paper, oneside, final, pagebackref]{amsart}
\usepackage[notref,notcite]{showkeys}
\usepackage[latin1]{inputenc}
\usepackage[T1]{fontenc}
\usepackage{amssymb}
\usepackage[stretch=10]{microtype}
\usepackage{mathtools}
\usepackage[DIV=11, headinclude=true]{typearea}
\usepackage[hidelinks, linktoc=all]{hyperref}

\providecommand{\href}[2]{#2}

\providecommand*{\backref}{}
\providecommand*{\backrefalt}{}
\renewcommand*{\backref}[1]{}
\renewcommand*{\backrefalt}[4]{%
	\ifcase #1 %
	\or
	  Cited page~#2.
	\else
	  Cited pages~#2.
	\fi
}

\newcommand\MTkillspecial[1]{
  \bgroup
  \catcode`\&=9
  \let\\\relax%
  \scantokens{#1}%
  \egroup
}
\newcommand\DeclarePairedDelimiterMultiline[3]{
  \DeclarePairedDelimiter{#1}{#2}{#3}
  \reDeclarePairedDelimiterInnerWrapper{#1}{star}{
    \mathopen{##1\vphantom{\MTkillspecial{##2}}\kern-\nulldelimiterspace\right.}
    ##2
    \mathclose{\left.\kern-\nulldelimiterspace\vphantom{\MTkillspecial{##2}}##3}}
}

\newcommand{\Pbb}{\mathbb{P}}
\newcommand{\Sbb}{\mathbb{S}}
\newcommand{\Z}{\mathbb{Z}}
\newcommand{\N}{\mathbb{N}}
\newcommand{\R}{\mathbb{R}}

\newcommand{\boC}{\mathcal{C}}
\newcommand{\dd}{\mathop{}\!\mathrm{d}}
\DeclarePairedDelimiterMultiline{\abs}{\lvert}{\rvert}
\DeclarePairedDelimiterMultiline{\norm}{\lVert}{\rVert}
\newcommand{\st}{\::\:}

\renewcommand{\epsilon}{\varepsilon}

\renewcommand{\phi}{\varphi}
\renewcommand{\leq}{\leqslant}
\renewcommand{\geq}{\geqslant}

\newtheorem{thm}{Theorem}[section]
\newtheorem{prop}[thm]{Proposition}

\newtheorem{lem}[thm]{Lemma}

\newtheorem*{prop*}{Proposition}

\theoremstyle{definition}

\newtheorem{rmk}[thm]{Remark}

\numberwithin{equation}{section}

\title{Growth of normalizing sequences in limit theorems for conservative maps}
\author{S\'ebastien Gou\"ezel}

\address{Laboratoire Jean Leray, CNRS UMR 6629,
Universit\'e de Nantes, 2 rue de la
Houssini\`ere,
44322 Nantes, France}
\email{sebastien.gouezel@univ-nantes.fr}

\date{\today}

\begin{document}

\begin{abstract}
We consider normalizing sequences that can give rise to nondegenerate limit
theorems for Birkhoff sums under the iteration of a conservative map. Most
classical limit theorems involve normalizing sequences that are polynomial,
possibly with an additional slowly varying factor. We show that, in
general, there can be no nondegenerate limit theorem with a normalizing
sequence that grows exponentially, but that there are examples where it
grows like a stretched exponential, with an exponent arbitrarily close to
$1$.
\end{abstract}

\maketitle

Let $T$ be a non-singular map on a measure space $(X,m)$, and $P$ a
probability measure which is absolutely continuous with respect to $m$. Given
a measurable function $f:X \to \R$, we say that its Birkhoff sums $S_n f =
\sum_{k=0}^{n-1} f \circ T^k$ satisfy a limit theorem if there is a sequence
$B_n \to \infty$ such that $S_n f/B_n$ converges in distribution with respect
to $P$ to a non-degenerate real random variable $Z$ (by non-degenerate, we
mean that $Z$ is not constant almost surely). In this case, the asymptotic
behavior of the sequence $B_n$ and the limiting variable $Z$ are defined
uniquely, up to scalar multiplication (see for instance~\cite[Theorem
14.2]{billingsley:book}).

There are many examples of such non-degenerate limit theorems in the
literature. Let us only mention the following ones:
\begin{enumerate}
\item Let $Y_0, Y_1,\dotsc$ be a sequence of i.i.d.~random variables which
    are in the domain of attraction of a stable law $Z$ of index $\alpha
    \in (0,2]$, and are centered if integrable. Then, for some slowly
    varying function $L$, one has the convergence $(Y_0+\dotsb+Y_{n-1})/
    (n^{1/\alpha} L(n)) \to Z$. This can be recast in dynamical terms as
    follows. Let $X = \R^{\N}$ with the measure $m = \Pbb_{Y_0}^{\otimes
    \N}$ and the left shift $T$. Define $f(x_0,x_1,\dotsc) = x_0$. Then
    $S_n f$ is distributed as $Y_0+\dotsb+Y_{n-1}$, and $S_n
    f/(n^{1/\alpha}L(n))$ converges to $Z$. For $\alpha=2$ and $L(n)=1$,
    this is an instance of the classical central limit theorem.
\item When $T$ preserves $m$ and $m$ has infinite mass, one obtains
    different limits. Let $\alpha \in [0,1)$ (we exclude $\alpha=1$ to
    avoid degenerate limit laws). The Darling-Kac theorem gives examples of
    transformations $T$ such that, for any function $f$ which is integrable
    and has non-zero average for $m$, then $S_n f/(n^\alpha L(n))$
    converges in distribution with respect to any probability measure $P\ll
    m$ to a Mittag-Leffler distribution of index $\alpha$ (where $L$ is a
    slowly varying function). See~\cite{aaronson_book, thaler_zweimuller}.
    This has been extended to functions with zero average by
    Thomine~\cite{thomine}, where the normalization becomes $(n^\alpha
    L(n))^{1/2}$ and the limit distribution is modified accordingly.
\item In the same setting as in the second example, take a function $f=1_E$
    where $E$ is a set of infinite measure. Thus, $S_n f/n$ records the
    proportion of time an orbit spends in $E$. Then one can devise such
    sets $E$ for which $S_n f/n$ converges in distribution to generalized
    arcine laws $\mathcal{L}_{\alpha,\beta}$, depending on $\alpha$ and an
    additional parameter $\beta \in (0,1)$. (We exclude $\beta=0$ and
    $\beta=1$ to ensure that $\mathcal{L}_{\alpha,\beta}$ is
    nondegenerate). See~\cite{thaler_zweimuller}.
\item Here is a nonergodic example. On $\Sbb^1 \times \Sbb^1$, consider the
    map $T(\alpha, x) = (\alpha, x + \alpha)$. Define a function $f(\alpha,
    x) = 1_{[0,r]}(x) - r$, where $r\in (0,1)$ is fixed. Then, with respect
    to Lebesgue measure, $S_n f/\log n$ converges in distribution to a
    Cauchy random variable. This theorem is due to
    Kesten~\cite{kesten_rotation}. It has been generalized recently to
    higher dimension by Dolgopyat and Fayad~\cite{dolgopyat_fayad1,
    dolgopyat_fayad2}: in dimension $d$, replacing the characteristic
    function of the interval by the characteristic function of a nice
    subset $\boC$ and adding a randomization over $r$, they obtain limit
    theorems with normalizations $n^{(d-1)/2d}$ or $(\log n)^d$ depending
    on the geometric characteristics of $\boC$.
\item It is possible to get any limiting distribution: Thouvenot and Weiss
    show in~\cite{thouvenot_weiss} that, for any probability-preserving
    system $T$, and for any real random variable $Z$, there exists a
    function $f$ such that $S_n f/n$ converges in distribution to $Z$.
    Moreover, if $Z$ is nonnegative, one can also take $f$ nonnegative if
    one is ready to replace the normalizing sequence by $n L(n)$ where $L$
    is slowly varying, see~\cite{aaronson_weiss_limits}.
\end{enumerate}

These examples show that the possible limit theorems are extremely diverse,
even in natural examples. However, in all these examples, the normalizing
sequence is of the form $n^\alpha L(n)$, where $\alpha \in [0,\infty)$ and
$L$ is a slowly varying sequence. All such $\alpha$ are realized in some
examples above. Our goal in this short note is to discuss how general this
restriction can be, and in particular the possible growth of a normalizing
sequence.

It is easy to see that, in probability-preserving systems, a normalizing
sequence has to grow at most polynomially. This is proved in
Proposition~\ref{prop:finite}. Our main result is that, in conservative
systems, a normalizing sequence can not grow exponentially
(Theorem~\ref{thm:infinite}), but that it can grow almost exponentially: we
exhibit examples of non-degenerate limit theorems with normalizing sequence
$e^{n^\alpha}$ for any $\alpha \in (0,1)$.

\section{The case of probability-preserving systems}

\begin{prop}
\label{prop:finite} Let $T$ be a probability-preserving map on a space $X$,
and $f:X\to \R$ a function such that $S_n f/B_n \to Z$ for some normalizing
sequence $B_n \to \infty$, where $Z$ is a real random variable which is not
almost surely $0$. Then $B_{n+1}/B_n \to 1$, and $B_n$ grows at most
polynomially, i.e., there exists $C$ such that $B_n=O(n^C)$.
\end{prop}
The proof is closely related to the proof of the normalizing constant
proposition in~\cite[Page~5]{aaronson_weiss_limits}
\begin{proof}
One has $S_n f/B_{n-1} = S_{n-1}f/B_{n-1} + f \circ T^{n-1}/B_{n-1}$.
Moreover, $f \circ T^{n-1}/B_{n-1}$ has the same distribution as $f/B_{n-1}$
as the probability is $T$-invariant. Since $B_n \to \infty$, one gets that
$f\circ T^{n-1}/B_{n-1}$ tends in probability to $0$. Since
$S_{n-1}f/B_{n-1}$ tends in distribution to $Z$, we deduce that $S_n
f/B_{n-1} \to Z$. Since $S_n f/B_n \to Z$ and $Z$ is non-degenerate, it
follows from the convergence of type theorem (see~\cite[Theorem
14.2]{billingsley:book}) that $B_n/B_{n-1} \to 1$.

Since $Z$ is not concentrated at $0$, there exist $\epsilon>0$ and $\alpha>0$
such that $\Pbb(\abs{Z}>\epsilon)>2\alpha$. Then, for large enough $n$, we
have
\begin{equation}
\label{eq:pupiouopiqusf}
\mu(\abs{S_n f}>B_n \epsilon)>3\alpha.
\end{equation}
Consider also $L$ such that $\Pbb(\abs{Z}\geq L) < \alpha$. Then, for large
enough, we have $\mu(\abs{S_n f} < L B_n) > 1 - \alpha$. Since the measure is
invariant, we also have $\mu(\abs{S_n f}(T^n x) < L B_n) > 1 - \alpha$. On
the intersection of these two events, $\abs{S_{2n} f} < 2LB_n$, and moreover
this happens with probability at least $1-2\alpha$. Since $\abs{S_{2n}f}>
B_{2n} \epsilon$ with probability at least $3\alpha$
by~\eqref{eq:pupiouopiqusf}, these two events intersect, and we obtain
$B_{2n}\epsilon \leq 2L B_n$. This shows that there exists a constant $C$
with $B_{2n} \leq C B_n$.

We have shown that $B_{2n}\leq C B_n$ and $B_{n+1} \leq C B_n$, for some
constant $C$. We deduce that $B_{2n+1} \leq C^2 B_n$ and $B_{2n} \leq C^2
B_n$. It follows by induction on $r$ that, for $n \in [2^r, 2^{r+1})$, one
has $B_n \leq C^{2r} B_1$, which is bounded by $2 B_1 \cdot n^{2 \log C/\log
2}$. This shows that $B_n$ grows at most polynomially.
\end{proof}
\begin{rmk}
The proof shows that the polynomial growth exponent for $B_n$ is bounded
solely in terms of the distribution of the limiting random variable $Z$.
\end{rmk}

\begin{rmk}
The assumption that $B_n \to \infty$ in Proposition~\ref{prop:finite} is
merely for convenience. Indeed, without this assumption, one still has that
$B_{n+1}/B_n$ tends to $1$ along integers $n$ where $B_n \to \infty$ (or
$B_{n+1} \to \infty$), with the same proof. In particular, there exists a
constant $D$ such that $B_{n+1} \leq 2 B_n$ if $B_{n+1} \geq D$. One deduces
that $B_n \leq D\cdot C^{2r}$ for $n \in [2^r, 2^{r+1})$, by separating the
case where $B_n < D$ (for which one can stop directly, without using the
induction) and $B_n \geq D$ (for which one uses the inductive assumption as
in the proof of Proposition~\ref{prop:finite}). This shows again that $B_n$
grows at most polynomially.
\end{rmk}

\section{Subexponential growth for normalizing sequences in conservative systems}

The proof of Proposition~\ref{prop:finite} relies crucially on the fact that
the measure is invariant and has finite mass. In the opposite direction, if
one does not assume any kind of recurrence, then one can get any behavior.
For instance, consider the right shift on $\Z$ and define the function $f(k)
= 2^k$. Then $S_n f(x)/2^n$ converges to $2^x$. This means that we get a
limit theorem, but which depends on the starting measure: Starting from
$\delta_0/3 + 2\delta_1/3$ or from $2\delta_0/3 + \delta_1/3$, say, we get
two different limit distributions, for the normalizing sequence $2^n$. For
another silly example, on $\Z \times \{-1,1\}$, let $T(n,a) = (n+1, a)$ and
$f(n,a) = a 2^{2^n}$. Then $S_n f/2^{2^{n-1}}$ converges in distribution with
respect to $P=(\delta_{(0,1)}+\delta_{(0,-1)})/2$ towards a Bernoulli random
variable. These examples show that there is nothing interesting to say in the
nonconservative case.

\medskip

When the map is conservative, on the other hand, we can use the
conservativity to obtain some rigidity. The next theorem shows in particular
that there can be no nontrivial limit theorem with a normalizing sequence
$2^n$.

\begin{thm}
\label{thm:infinite} Let $T$ be a conservative map on the measure space
$(X,m)$. Consider a probability measure $P$ which is absolutely continuous
with respect to $m$, and a measurable function $f$. Assume that $S_n f/B_n$
converges in distribution with respect to $P$ towards a limiting random
variable $Z$ which is not almost surely $0$. Then, for any $\rho>0$, one has
$B_n = o(e^{\rho n})$.
\end{thm}

\begin{rmk}
As a test case for the usability of proof assistants for current mathematical
research, Theorem~\ref{thm:infinite} and its proof given below have been
completely formalized and checked in the proof assistant Isabelle/HOL, see
the file \verb+Normalizing_Sequences.thy+ in~\cite{Ergodic_Theory-AFP2018}.
In particular, the correctness of this theorem is certified.
\end{rmk}

The main intuition behind the proof of Theorem~\ref{thm:infinite} is the
following. Assume by contradiction that $B_{n+1}$ is much larger than $B_n$.
Then
\begin{equation}
\label{eq:oiwuxcvpoiuop}
S_{n+1} f(x)/B_{n+1} = f(x)/B_{n+1} + S_n f(Tx)/B_n \cdot B_n/B_{n+1}.
\end{equation}
The term $f(x)/B_{n+1}$ is small in distribution. Moreover, $S_n f/B_n$ is
tight, so when multiplied by $B_n/B_{n+1}$ one should get something small,
and get a contradiction with the fact that the limit $Z$ is not a Dirac mass
at $0$. There is a difficulty: in~\eqref{eq:oiwuxcvpoiuop}, there is an
additional composition by $T$ in $S_n f(T x)/B_n$. Still, this heuristic
argument can be made into a rigorous proof that $B_{n+1} \leq C \max_{i \leq
n} B_i$, for some $C>0$, in Lemma~\ref{lem:C}. This excludes superexponential
growth for $B_n$, but not exponential growth. To improve this bound, we will
take advantage of the opposite decomposition $S_{n+1}f = S_n f+ f\circ T^n$,
to obtain another growth control in Lemma~\ref{lem:Bnk}, that only holds
along a subsequence of density $1$ but in which one loses a multiplicative
constant $L$ that only depends on the distribution of $Z$. The point is that
this estimate also applies to the system $T^j$ and the sequence $B_{jn}$ for
any $j$, replacing in essence $L$ with $L^{1/j}$ and making it arbitrarily
close to $1$. The theorem then follows from the combination of these two
lemmas.

An important tool in the proof is the transfer operator $\hat T$, i.e., the
predual of the composition by $T$, acting on $L^1(m)$. It satisfies $\int f
\cdot g\circ T \dd m = \int \hat T f\cdot g \dd m$. We will use the following
characterizations of conservativity: for any nonnegative $f$, then
$\sum_{n=0}^\infty f(T^n x)$ and $\sum_{n=0}^\infty \hat T^n f(x)$ are both
infinite for almost every $x$ with $f(x)>0$, see Propositions~1.1.6 and~1.3.1
in~\cite{aaronson_book}.

We start with a lemma asserting that the growth from $B_n$ to $B_{n+1}$ is
uniformly bounded.

\begin{lem}
\label{lem:C} Under the assumptions of Theorem~\ref{thm:infinite}, there
exists $C$ such that, for all $n$, one has $B_{n+1} \leq C \max(B_0,\dotsc,
B_n)$.
\end{lem}
The idea of the proof is that $S_n f(x)/B_n = S_k f(x)/B_n + S_{n-k}f(T^k
x)/B_n$. The first term is small if $n$ is large and $k$ is fixed, while the
second term should be small with high probability if $B_n/B_{n-k}$ is large,
as $S_{n-k} f/B_{n-k}$ is distributed like $Z$ with respect to $P$. The
difficulty is that there is a composition with $T^k$, under which $P$ is not
invariant, so one can not argue that $S_{n-k} f(T^k x)/B_{n-k}$ is
distributed like $Z$. But one can still take advantage of the conservativity
to make the argument go through.
\begin{proof}
Replacing $m$ with $m+P$, we can assume without loss of generality that $h
=\dd P/\dd m$ is bounded by $1$. We can also assume that
$\max(B_0,\dotsc,B_n)$ tends to infinity, since the conclusion is obvious
otherwise. Multiplying $f$ and $Z$ by a constant, we can also assume
$\Pbb(\abs{Z}>2)>3\alpha>0$ as $Z$ is not concentrated at $0$.

By conservativity, $\sum_{k=1}^K h(T^k x)$ tends $m$-almost everywhere to
$+\infty$ on the set $\{h>0\}$, and therefore $P$-almost everywhere.
Therefore, we can fix $K$ such that $P(\sum_{k=1}^K h(T^k x) \geq 1) >
1-\alpha$.

Let $\delta>0$ be small enough so that $K\delta < \alpha$. Let $\epsilon >0$
be small enough so that, for any $k\leq K$ and for any measurable set $U$
with $P(U)<\epsilon$, one has
\begin{equation}
  \label{eq:delta}
  \int 1_U(x) \hat T^k h(x) \dd P(x) < \delta.
\end{equation}
This is possible since the function $\hat T^k h(x)$ is integrable with
respect to $m$ (its integral is $\int h \dd m = P(X) = 1$), and therefore
with respect to $P$ as $P \leq m$.

Let $C>0$ be such that $\Pbb(\abs{Z}\geq C) < \epsilon$. We will show that
$B_n \leq C \max_{i<n} B_k$ for large enough $n$, by contradiction. Assume
instead that $n$ is large and $B_n > CB_k$ for any $k<n$. In particular,
since $\max(B_0,\dotsc, B_{n-1})$ tends to infinity with $n$, this implies
that $B_n$ is very large. With probability at least $3\alpha$, we have
$\abs{S_n f/B_n}>2$, as $\Pbb(\abs{Z} > 2) > 3\alpha$ and $S_n f/B_n$ tends
in distribution to $Z$. Moreover, $P(\forall k\leq K, \abs{S_k f}/B_n \leq
1)$ is arbitrarily close to $1$, as $B_n$ is very large while $K$ is fixed.
In particular, it is $\geq 1-\alpha$. This shows that the event
\begin{equation*}
  V =  \left\{x \st \sum_{k=1}^K h(T^k x) \geq 1 \text{ and }\abs{S_n f(x)/B_n}>2 \text{ and }
  \forall k\in [1, K], \abs{S_k f(x)}/B_n \leq 1\right\}
\end{equation*}
has probability at least $\alpha$, as it the intersection of three sets whose
measures are at least $1-\alpha$ and $3\alpha$ and $1-\alpha$. For $x\in V$,
we have for any $k\in [1, K]$ the inequality $\abs{S_{n-k} f(T^k x)}/B_n \geq
1$, and therefore $\abs{S_{n-k} f(T^k x)}/B_{n-k} > C$. We get
\begin{equation*}
  \alpha \leq P(V) = \int 1_V(x) \dd P(x) \leq
  \int \left(\sum_{k=1}^K h(T^k x) 1_{\abs{S_{n-k} f}/B_{n-k} > C}(T^k x)\right) \dd P(x).
\end{equation*}
Writing $\dd P = h\dd m$ and changing variables by $y= T^k x$ in the $k$-th
term of the sum, we obtain
\begin{align*}
  \alpha &\leq \sum_{k=1}^K \int h(y)1_{\abs{S_{n-k} f}/B_{n-k} > C}(y)\cdot \hat T^k h(y) \dd m(y)
  \\&
  = \sum_{k=1}^K \int 1_{\abs{S_{n-k} f}/B_{n-k} > C}(y) \cdot \hat T^k h(y) \dd P(y).
\end{align*}
As $\Pbb(\abs{Z}\geq C) < \epsilon$, we have $P(\abs{S_m f}/B_m > C) <
\epsilon$ for large enough $m$. Then the property of $\delta$ given
in~\eqref{eq:delta} ensures that $\int 1_{\abs{S_m f}/B_m
> C}(y) \cdot \hat T^k h(y) \dd P(y) \leq \delta$. This gives
\begin{equation*}
  \alpha \leq \sum_{k=1}^K \delta = K\delta.
\end{equation*}
This is a contradiction as $K\delta<\alpha$ by construction.
\end{proof}
If we could show that $C$ in Lemma~\ref{lem:C} can be taken arbitrarily close
to $1$, then Theorem~\ref{thm:infinite} would follow directly. However, we do
not know if this is true. What we will show instead is that such an
inequality is true along a sequence of integers of density $1$. This will be
enough to conclude the proof. To proceed, we will need the following
technical lemma, relying on conservativity.
\begin{lem}
\label{lem:measure_disjoint} Let $T$ be a conservative dynamical system on
the measure space $(X,m)$. Consider a finite measure $P$ which is absolutely
continuous with respect to $m$, and disjoint measurable sets $A_n$. Then
$P(T^{-n}A_n)$ tends to $0$ along a set of integers of density $1$.
\end{lem}
\begin{proof}
Replacing $m$ by a finite measure which is equivalent to it, we can assume
that $m$ is finite. Replacing $P$ by $P+m$ (which only makes the conclusion
stronger), we can also assume that $P$ is equivalent to $m$. Finally,
replacing $m$ by $P/P(X)$, we can even assume that $P=m$, and that it is a
probability measure.

Let $\epsilon>0$. Let $L$ be a large constant. Let $\hat T$ be the transfer
operator associated to $T$, i.e., the adjoint (for the measure $m$) of the
composition by $T$. We have $\sum_{j=0}^{k-1} \hat T^j 1(x) \to \infty$
almost everywhere, by conservativity. In particular, one can find a set $U$
with $m(X\setminus U) \leq \epsilon/2$ and an integer $K$ such that, for all
$x\in U$, one has $\sum_{j=0}^{K-1} \hat T^j 1(x) \geq L$. This gives in
particular for all $n$
\begin{equation}
\label{eq:iopusqpoiusdqf}
  m(T^{-n} A_n) \leq \epsilon/2 + m(T^{-n}A_n \cap U).
\end{equation}

Let $n\geq K$. Since the sets $A_{n-i}$ are disjoint, we have
$\sum_{i=0}^{K-1} 1_{A_{n-i}} (T^n x) \leq 1$. Integrating this inequality,
and applying the transfer operator $\hat T^i$ to the $i$-th term in the sum,
we obtain
\begin{equation*}
  \sum_{i=0}^{K-1} \int \hat T^i 1(x) \cdot 1_{A_{n-i}} (T^{n-i} x) \dd m(x)\leq 1.
\end{equation*}
Let us sum this inequality over $n \in [K,N-1]$. Changing variables $p=n-i$
and discarding the terms with $p<K$ or $p>N-K$, we obtain
\begin{equation*}
  \sum_{p=K}^{N-K} \int \left(\sum_{i=0}^{K-1} \hat T^i 1(x)\right) 1_{A_p}(T^p x) \dd m(x) \leq N.
\end{equation*}
Since $\left(\sum_{i=0}^{K-1} \hat T^i 1(x)\right) \geq L$ on $U$, this gives
\begin{equation*}
  \sum_{p=K}^{N-K} m(T^{-p} A_p \cap U) \leq N/L.
\end{equation*}
Therefore, thanks to~\eqref{eq:iopusqpoiusdqf}
\begin{equation*}
  \sum_{p=0}^{N-1} m(T^{-p} A_p) \leq N \epsilon/2 + 2K + \sum_{p=K}^{N-K} m(T^{-p} A_p \cap U)
  \leq N \epsilon/2 + 2K + N/L.
\end{equation*}
If $L$ is large enough, this is bounded by $N\epsilon$ for large $N$. We have
proved that $m(T^{-p}A_p)$ tends to zero in Cesaro average. This is
equivalent to convergence to zero along a density one set of integers (see
for instance~\cite[Theorem~1.20]{walters}).
\end{proof}

We can take advantage of the previous lemma to obtain a bound on
$B_{n+1}/B_n$ which is only true along a subsequence of density $1$, but in
which we only lose an explicit multiplicative constant. The difference
between this lemma and Lemma~\ref{lem:C} is that the constant $L$ below does
only depend on the distribution of $Z$, contrary to the constant $C$ of
Lemma~\ref{lem:C}. This means that it will be possible to apply this lemma to
an iterate $T^j$ of $T$ with the same $L$, replacing in essence $L$ with
$L^{1/j}$ and making it arbitrarily close to $1$.

\begin{lem}
\label{lem:Bnk} Let $T$ be a conservative dynamical system on the measure
space $(X,m)$. Consider a probability measure $P$ which is absolutely
continuous with respect to $m$, and two measurable functions $f$ and g.
Assume that $(g+S_n f)/B_n$ converges in distribution with respect to $P$
towards a limiting random variable $Z$ which is not concentrated at $0$. Then
there exists $L>1$ depending only on the distribution of $Z$ such that
$B_{n+1} \leq L \max(B_0,\dotsc, B_n)$ along a set of integers of density
$1$.
\end{lem}
\begin{proof}
Multiplying $Z$, $f$ and $g$ by a suitable constant if necessary, we can
assume that $\Pbb(\abs{Z}>2) = 3\alpha>0$. Choose $L>3$ such that
$\Pbb(\abs{Z}\geq L-1) < \alpha$. Note that $L$ only depends on the
distribution of $Z$. Let $M_n = \max(B_0,\dotsc, B_n)$. Let us show that $\{n
\st B_{n+1} >L M_n\}$ has zero density. Let $A_n = \{x \st \abs{f(x)} \in
(LM_n, LM_{n+1})\}$. These sets are disjoint as $M$ is nondecreasing.
Consider $n$ such that $B_{n+1} > LM_n$.

Since $(g+S_nf)/B_n$ has a distribution close to that of $Z$, one has
$P(\abs{(g+S_nf)/B_n} < L) \geq 1-\alpha$. Therefore,
$P(\abs{(g+S_nf)/B_{n+1}} < 1) \geq 1-\alpha$ since $B_{n+1} > L B_n$. On the
other hand, $P(\abs{(g+S_{n+1} f)/B_{n+1}} \in [2, L-1]) \geq 3\alpha-\alpha
= 2\alpha$. Hence,
\begin{equation}
\label{eq:uopiwxucvioupxwcv}
  P\left\{x \st \abs{(g(x)+S_n f(x))/B_{n+1}} < 1
  \text{ and }\abs{(g(x)+S_{n+1} f(x))/B_{n+1}} \in [2, L-1]\right\} \geq \alpha.
\end{equation}
On this event, $f \circ T^n / B_{n+1} = (g+S_{n+1} f)/B_{n+1} -(g+S_n
f)/B_{n+1}$ has an absolute value in $(1, L)$, i.e., $f(T^n x) \in (B_{n+1},
LB_{n+1})$. Since $B_{n+1} > LM_n$, this shows that $x \in T^{-n} A_n$.
Hence, the event in~\eqref{eq:uopiwxucvioupxwcv} is contained in $T^{-n}
A_n$. Since the probability of $T^{-n}A_n$ tends to $0$ along a set of
density $1$ by Lemma~\ref{lem:measure_disjoint}, this shows that the
inequality $B_{n+1}>LM_n$ can only hold along a set of density $0$.
\end{proof}

We can now conclude the proof of Theorem~\ref{thm:infinite}.

\begin{proof}[Proof of Theorem~\ref{thm:infinite}]
Let $j>0$ and $k\in [0,j)$. We claim that, along a subsequence of density
$1$, we have
\begin{equation}
\label{eq:iwucvxipouwxcv}
  B_{nj+k} \leq  L \max(B_k, B_{j+k},\dotsc, B_{(n-1)j+k}),
\end{equation}
where $L=L(Z)$ is given by Lemma~\ref{lem:Bnk}. Indeed, consider the
conservative dynamical system $\tilde T = T^j$, the function $\tilde f = (S_j
f) \circ T^k$, and the function $\tilde g = S_k f$. Denoting by $\tilde S$
the Birkhoff sums for the transformation $\tilde T$, we have $\tilde g +
\tilde S_n \tilde f = S_{nj+k}f$. Therefore, $(\tilde g + \tilde S_n \tilde
f)/B_{nj+k}$ converges in distribution to $Z$. Lemma~\ref{lem:Bnk} then
implies~\eqref{eq:iwucvxipouwxcv}.

Define $M_n = \max(B_0,B_1,\dotsc, B_{(n+1)j-1})$. Along a subsequence of
density $1$, all the inequalities~\eqref{eq:iwucvxipouwxcv} hold for $k\in
[0,j)$. We get that $M_{n+1} \leq L M_n$ along a subsequence of density $1$.
Moreover, Lemma~\ref{lem:C} shows that $M_{n+1} \leq C^j M_n$ for all $n$. We
obtain $M_n \leq L^n \cdot (C^j)^{o(n)} M_0$. Therefore, for large enough
$n$, we have $M_n \leq L^{2n} M_0$. In particular, for any $m \in [nj,
(n+1)j)$, we get
\begin{equation*}
  B_m \leq M_n \leq L^{2n} M_0 \leq L^{2m/j} M_0.
\end{equation*}
The exponential growth rate of the term on the right hand side, in terms of
$m$, can be made arbitrarily small by taking $j$ large enough.
\end{proof}

\section{An example with a stretched exponential normalizing sequence}

Let us now describe an example showing that Theorem~\ref{thm:infinite}, which
excludes the exponential growth of normalizing sequences, is almost sharp.
Indeed, for any $\alpha<1$, we construct a conservative dynamical system
(more specifically a Markov chain, which preserves an infinite measure) and a
function $f$ such that $S_n f/e^{n^\alpha}$ converges in distribution to a
non-degenerate limit.

Let
\begin{equation}
\label{eq:distrib_p}
  p_n = c/(n (\log n)^2)
\end{equation}
for $n\geq 2$ and $p_0=p_1=0$, where $c$ is adjusted so that $p$ is a
probability distribution. We consider a recurrent Markov chain on a set $A$
for which the excursion away from a reference point $0$ has length $n$ with
probability $p_n$, and for which one can define the height of a point, i.e.,
how many steps away from $0$ it is.

The simplest such example is on the set $A=\{0\} \cup \{(n, i) \st n
> 1, i \in [1,n-1]\}$ in which one jumps from $0$ to $(n,1)$ with probability
$p_n$, from $(n,i)$ to $(n, i+1)$ with probability $1$ if $i<n-1$, and from
$(n,n-1)$ to $0$ with probability $1$. The height of $(n,i)$ is $i$. One can
also realize such an example on $A=\N$, with jumps from $n$ to $0$ with
probability $q_n$ and from $n$ to $n+1$ with probability $1-q_n$, where $q_n$
is defined inductively by the relation $p_{n+1} = (1-q_0)\dotsm (1-q_{n-1})
q_n$. The height of $n$ is just $n$ in this example.

Let $X$ be the space of possible trajectories of the Markov chain, $m$ the
corresponding (infinite) measure on trajectories, $\pi:X\to A$ the projection
associating to a trajectory its starting point, and $T$ the shift map
forgetting the first point of the trajectory. It is conservative as the
Markov chain is recurrent. Abusing notations, if $g$ is a function on $A$, we
will write $g(x)$ for $g(\pi x)$. For instance, $h(x)$ will denote the height
of the starting point of the trajectory $x$.

Let $P$ be the Markov measure on paths starting from $0$. Then, with respect
to $P$, the height $h(T^n x)$ is by definition distributed like $h(X_n)$
where $X_n$ is the random walk starting from $0$. Since $p_n$
in~\eqref{eq:distrib_p} tends very slowly to $0$, one expects most excursions
away from $0$ to be extremely long, so that $h(X_n)$ should be of the order
of magnitude of $n$ with large probability. This is true, but we will need
the following more precise estimate.

\begin{lem}
\label{lem:distrib_h} When $k\to \infty$ and $n-k \to \infty$, one has
\begin{equation*}
  P(h(T^n x) = k) \sim \frac{1}{(\log k) \cdot (n-k)}.
\end{equation*}
\end{lem}
\begin{proof}
Denote by $Y_i$ the distribution of the length of the $i$-th excursion away
from $0$ of the random walk. These variables are independent, and share the
same distribution $Y$ given by $\Pbb(Y=s)=p_s$. Then $X_n$ is at height $k$
exactly if there is an index $j$ such that $Y_1+\dotsb+Y_j = n-k$ and
moreover $Y_{j+1}>k$. Since the distribution of the excursion after $n-k$ is
independent of the fact that $n-k$ was reached by such a sum, we get
\begin{equation*}
  P( h(T^n x) = k) = \Pbb(Y>k)\cdot \Pbb(\exists j, Y_1+\dotsb+Y_j = n-k).
\end{equation*}
The last event in this equation is a \emph{renewal event}. Local asymptotics
of renewal probabilities are known for $p$ distributed as
in~\eqref{eq:distrib_p}: by~\cite{nagaev_no_moment} (see
also~\cite{alexander_berger_no_moment}), one has when $s\to \infty$
\begin{equation*}
  \Pbb(\exists j, Y_1+\dotsb+Y_j = s) \sim \frac{\Pbb(Y = s)}{\Pbb(Y>s)^2}.
\end{equation*}
Since $\Pbb(Y>s) \sim c/\log s$, we obtain
\begin{equation*}
  P( h(T^n x) = k) \sim \frac{c}{\log k} \cdot \frac{c/ (n-k) (\log (n-k))^2}{(c/\log (n-k))^2}
  = \frac{1}{(\log k) \cdot (n-k)}.
  \qedhere
\end{equation*}
\end{proof}

\begin{lem}
\label{lem:limit_distrib_h} Let $\beta \in (0,1)$. When $n$ tends to
infinity,
\begin{equation*}
P(h(T^n x) \in [n-n^\beta, n)) \to \beta.
\end{equation*}
\end{lem}
\begin{proof}
Let us first note that, with high probability, $k=h(T^n x)$ is bounded away
from $0$ and $n$. Indeed, if $C>0$ is fixed, the approximations for the
probabilities in Lemma~\ref{lem:distrib_h} satisfy
\begin{equation*}
  \sum_{k=n/2}^{n-C} \frac{1}{(\log k) \cdot (n-k)}
  \sim \frac{1}{\log n} \sum_{k=n/2}^{n-C} \frac{1}{n-k}
  \sim \frac{1}{\log n} \sum_{m=C}^{n/2} \frac{1}{m} \to 1.
\end{equation*}
Thanks to Lemma~\ref{lem:distrib_h}, this shows that the contribution of this
range of values of $h(T^n x)$ has probability close to $1$. Therefore, we may
without loss of precision replace $P(h(T^n x)=k)$ with the approximating
probabilities given in this lemma, and even by $1/((\log n) \cdot (n-k))$.

Therefore,
\begin{equation*}
  P(h(T^n x) \in [n-n^\beta, n))
  \sim \frac{1}{\log n} \sum_{k=n-n^\beta}^{n-1} \frac{1}{n-k}
  = \frac{1}{\log n} \sum_{m=1}^{n^\beta} \frac{1}{m}
  \sim \frac{\log(n^\beta)}{\log n} = \beta.
\qedhere
\end{equation*}
\end{proof}

Let $\alpha \in (0,1)$. Let $s(n) = e^{(n+1)^\alpha} - e^{n^\alpha} \geq 0$.
It is equivalent to $\tilde s(n)=\alpha n^{\alpha-1} e^{n^\alpha}$ at
infinity, so we could use $\tilde s$ instead of $s$ in what follows, but the
point of the construction is clearer with $s$. Define a function $f$ by $f(x)
= s(h(x)) \geq 0$. In this way, the sum of the values of $f$ along an
excursion up to height $k-1$ is $e^{k^\alpha}-1$.

\begin{thm}
\label{thm:large_normalization} The sequence $S_n f(x)/ e^{n^\alpha}$
converges in distribution with respect to $P$ towards a non-trivial random
variable, equal to $0$ with probability $\alpha$ and to $1$ with probability
$1-\alpha$.
\end{thm}
\begin{proof}
Let $\epsilon>0$. We claim that, if $h(T^n x) \in [n-n^{1-\alpha-\epsilon},
n)$ (which happens asymptotically with probability $1-\alpha-\epsilon$ by
Lemma~\ref{lem:limit_distrib_h}) then $S_n f(x)/e^{n^\alpha}$ is close to
$1$. Moreover, we claim that, if $h(T^n x) \in [n-n^{1-\epsilon},
n-n^{1-\alpha+\epsilon}]$ (which happens asymptotically with probability
$(1-\epsilon)-(1-\alpha+\epsilon) = \alpha -2\epsilon$ by
Lemma~\ref{lem:limit_distrib_h}) then $S_n f(x)/e^{n^\alpha}$ is close to
$0$. As the probabilities of these two events almost add up to one (up to
$3\epsilon$), this shows the desired convergence in distribution by letting
$\epsilon$ tend to $0$.

It remains to prove the claims. For this, consider $k = n-n^\beta$ for some
$\beta$. Then
\begin{equation*}
k^\alpha = n^\alpha (1-n^{\beta-1})^\alpha = n^\alpha -\alpha n^{\alpha+\beta-1}(1+o(1)).
\end{equation*}
Therefore,
\begin{equation}
\label{eq:ekalpha}
  \frac{e^{k^\alpha}}{e^{n^\alpha}} = e^{-\alpha n^{\alpha+\beta-1}(1+o(1))}
  \to \begin{cases} 0 & \text{if } \beta>1-\alpha,\\ 1 & \text{if }\beta < 1-\alpha.\end{cases}
\end{equation}

Let us now prove the claims. Assume first that $h(T^n x) \in
[n-n^{1-\alpha-\epsilon},n)$. Then the sum of the values of $f$ along the
last excursion up to time $n-1$ is equal to $\sum_{j=0}^{h(T^n x)-1} s(j) =
e^{h(T^n x)^\alpha} - 1$. By~\eqref{eq:ekalpha} with
$\beta<1-\alpha-\epsilon$, we deduce that this sum divided by $e^{n^\alpha}$
is close to $1$. We should then add the contributions of the first
excursions. Since their total length is bounded by $n^{1-\alpha-\epsilon}$,
the maximal height they could have reached is $n^{1-\alpha-\epsilon}$, and
their total contribution is at most $n \cdot
e^{(n^{1-\alpha-\epsilon})^\alpha}$. This is negligible. This proves the
first claim.

The second claim, for $h(T^n x) \in [n-n^{1-\epsilon}, n -
n^{1-\alpha+\epsilon}]$, is proved analogously. Indeed, by~\eqref{eq:ekalpha}
with $\beta > 1-\alpha+\epsilon$, the contribution of the last excursion is
negligible. And the contribution of the other excursions is bounded by $n
e^{n^{(1-\epsilon) \alpha}}$ and is also negligible. This concludes the
proof.
\end{proof}

\begin{rmk}
In our example, the sequence $p_n=c/(n (\log n)^2)$ tends very slowly to $0$.
If one tries to build a less extreme example, by taking $p_n = c/n^\gamma$
for some $\gamma>1$, then the construction fails. Indeed, one can check that
in this case $h(T^n x)$ is distributed over the whole interval $[0,n]$
(instead of being very concentrated around $n$). More precisely, $h(T^n x)/n$
converges in distribution to a random variable with support equal to $[0,1]$
(this follows from direct computations, or from the arcine law for waiting
times~\cite{thaler_zweimuller}). Then $e^{h(T^n x)^\alpha} / e^{n^\alpha}$
converges in distribution to $0$, giving a degenerate limit theorem.
\end{rmk}

\begin{rmk}
\label{rmk:qsidufpioqsdf} The above example is quite flexible. For instance,
instead of obtaining the limit $\alpha \delta_0 + (1-\alpha)\delta_1$, one
can obtain any limiting distribution $Z$ which is equal to $0$ with
probability $\alpha$ and to $Z'$ with probability $(1-\alpha)$ where $Z'$ is
an arbitrary real random variable, with an ergodic conservative map, as
follows. We describe it in the setting of the random walk on $\N$ above, with
return probabilities $q_n$ from $n$ to $0$. Replace the space $\N$ of the
random walk with $\N \times \R$, with the following transition probabilities.
Jump from $(n,t)$ to $(n+1,t)$ with probability $1-q_n$, and from $(n,t)$ to
$\{0\} \times \R$ with probability $q_n \otimes \Pbb_{Z'}$ where $\Pbb_{Z'}$
is the distribution of $Z'$. If one starts from $P = \delta_0 \otimes
\Pbb_{Z'}$, then the first component of this new random walk is distributed
like the original random walk on $\N$, and the second component is always
distributed like $Z'$, with independence every times the walk returns to $0$.
Define a function $g(n,t) = t e^{n^\alpha}$, and let $f(x)= g(T x)-g(x)$.
Then the Birkhoff sum $S_n f(x)$ is equal to $t e^{k^\alpha} - g(x)$ if
$\pi(T^n x) = (k, t)$. The proof of Theorem~\ref{thm:large_normalization}
then shows that $S_n f/e^{n^\alpha}$ converges to $Z$.
\end{rmk}


\begin{rmk}
One can also construct an example where $S_n f/B_n$ converges to a nontrivial
distribution, but $B_{n+1}/B_n$ does not tend to $1$. Consider the same
example as in Theorem~\ref{thm:large_normalization}, but where the
probabilities $p_n$ are nonzero only for even $n$: there is a periodicity
phenomenon in this Markov chain, of which we will take advantage. Define
$g(n) = e^{n^\alpha}$ for even $n$, and $g(n) = e^{n^\alpha}/2$ for odd $n$.
Let $f(x) = g(Tx) - g(x)$. Then $S_n f/B_n$ converges in distribution with
respect to $P$ towards $Z = \alpha \delta_0 + (1-\alpha)\delta_1$, where $B_n
= e^{n^\alpha}$ for even $n$ and $B_n = e^{n^\alpha}/2$ for odd $n$. In this
example, if one denotes by $Q$ the measure on trajectories starting at $1$,
then $S_n f/B_n$ does not converge to $Z$ with respect to $Q$, while it does
with respect to $P$, whereas $P$ and $Q$ are both absolutely continuous with
respect to the invariant infinite measure $m$. This phenomenon does not
happen for probability-preserving maps, by Eagleson's
Theorem~\cite{eagleson}. This shows that Eagleson's Theorem is only valid in
full generality for probability-preserving maps, and not for conservative
maps in general. Zweim\"uller has proved in~\cite{zweimuller_mixing} that
Eagleson's Theorem holds for Birkhoff sums in conservative maps under an
additional assumption of asymptotic invariance, which in our setting
translates to the fact that $f\circ T^n/B_n$ converges to $0$ in
distribution. Indeed, this is not the case in the previous counterexample.
\end{rmk}

\bibliography{biblio}
\bibliographystyle{amsalpha}
\end{document}